\documentclass[12 pt]{amsart}
\title[On the log HdR spectral sequence for K3 surfaces]{On the logarithmic Hodge-de Rham spectral sequence for curves on K3 surfaces}
\author{Daniel Bragg}
\usepackage{macros}
\usepackage{mathtools}

\begin{document}

\begin{abstract}
    We show that if $X$ is a supersingular K3 surface then there exists a curve $D$ on $X$ such that the logarithmic Hodge--de Rham spectral sequence for $(X,D)$ is nondegenerate.
\end{abstract}

\maketitle

%\tableofcontents

\section{Introduction}

For a smooth projective variety $X$ over a field of characteristic $0$ the Hodge--de Rham spectral sequence
\[
    \mathrm{E}_1^{a,b}=\H^b(X,\Omega^a_X)\implies\H^{a+b}(X,\Omega^{\bullet}_X)=:\H^{a+b}_{\dR}(X)
\]
is degenerate. This need not be the case in positive characteristic, but nevertheless remains true in some important situations. A famous result of Deligne--Illusie gives degeneracy in degrees $<p$ if $X$ lifts over the ring of truncated Witt vectors $W_2$ \cite{MR894379}. Rudakov--Safarevic showed that if $X$ is a K3 surface in positive characteristic then the HdR spectral sequence is degenerate \cite{MR460344}. Shorter proofs were later found by Lang and Nygaard \cite{MR554382,MR585962}. This result is most difficult when $X$ is supersingular, and in this case it seems that the Hodge--de Rham spectral sequence is only barely degenerate, as various closely related spectral sequences tend to be nondegenerate, for example the slope spectral sequence \cite[II 7.2]{Ill79} and the conjugate spectral sequence \cite[IV 2.17]{illusie_raynaud}. In this paper we consider the logarithmic variant of the Hodge--de Rham spectral sequence. We say a divisor $D\subset X$ has \emph{strict normal crossings} (\emph{is snc}) if the irreducible components of $D$ are smooth and intersect transversely. Given an snc divisor $D\subset X$ there is an associated logarithmic de Rham complex $\Omega^{\bullet}_X(\log D)$ and a logarithmic Hodge--de Rham spectral sequence
\[
    \mathrm{E}_1^{a,b}=\H^b(X,\Omega^{a}_{X}(\log D))\implies\H^{a+b}(X,\Omega^{\bullet}_X(\log D)).
\]
This behaves similarly to the ordinary Hodge--de Rham spectral sequence in many ways; for instance, it is always degenerate in characteristic 0, and there is a logarithmic variant of Deligne--Illusie \cite[4.2]{MR894379}, \cite[10.21]{MR1193913}. However, we will show that if $X$ is a supersingular K3 surface then we may choose $D$ so that this spectral sequence is non-degenerate. More precisely, we characterize the divisors $D$ with this property as follows. We recall two discrete invariants attached to a K3 surface $X$ over an algebraically closed field $k$ of positive characteristic $p$. The first is the \emph{height}, denoted $h$, which is either an integer $1\leq h\leq 10$, or is equal to $\infty$. In the former case we say that $X$ has \emph{finite height}, in the latter case we say that $X$ is \emph{supersingular}. By the solution of the Tate conjecture for K3 surfaces, it is known that supersingularity is equivalent to $X$ having Picard number 22. The second is the \emph{Artin invariant}, denoted $\sigma_0$. If $X$ has finite height we take the convention that $\sigma_0=\infty$. If $X$ is supersingular then $\sigma_0$ is an integer satisfying $1\leq\sigma_0\leq 10$, characterized by the property that
\[
    \mathrm{disc}(\Pic(X))=-p^{2\sigma_0}.
\]
If $X$ is supersingular then the discriminant group of $\Pic(X)$ is killed by $p$, hence is an $\bF_p$-vector space of dimension $2\sigma_0$. It is naturally identified with the $\bF_p$-subspace $\N_0\subset\Pic(X)\otimes_{\bZ}\bF_p$ that is the image modulo $p$ of the subgroup of classes $L\in\Pic(X)$ having the property that $L.M$ is divisible by $p$ for all $M\in\Pic(X)$. For an snc divisor $D\subset X$ with irreducible components $D_1,\ldots,D_r$, we let $\N_D\subset\Pic(X)\otimes\bF_p$ denote the $\bF_p$-vector space generated by the classes of the $\ms O(D_i)$. Here is what we will prove:

\begin{theorem}\label{thm:main thm 1}
    Let $X$ be a K3 surface and $D\subset X$ an snc divisor.% The logarithmic Hodge--de Rham spectral sequence for $(X,D)$ is degenerate if and only if $\dim_{\bF_p}(\T_D\cap\N_0)<\sigma_0$. In particular, if $X$ has finite height the log HdR spectral sequence is always degenerate, and if $X$ is supersingular it is nondegenerate if $\dim_{\bF_p}(\T_D\cap\N_0)\geq\sigma_0$. Moreover, for any supersingular $X$ and any $r\geq\sigma_0$ there exists an snc divisor with exactly $r$ irreducible components having this property.
    %and let $\T_D\subset\Pic(X)\otimes\bF_p$ be the $\bF_p$-vector space generated by the classes of the $\ms O(D_i)$.
    \begin{enumerate}
        \item If $X$ has finite height then the logarithmic Hodge--de Rham spectral sequence for $(X,D)$ is degenerate.
        \item If $X$ is supersingular of Artin invariant $\sigma_0$, then the logarithmic Hodge--de Rham spectral sequence for $(X,D)$ is nondegenerate if and only if $\dim_{\bF_p}(\N_D\cap\N_0)\geq\sigma_0$. Moreover, for any supersingular $X$ and any $r\geq\sigma_0$ there exists an snc divisor with exactly $r$ irreducible components having this property.
        % \item If $r<\sigma_0$ then the logarithmic Hodge--de Rham spectral sequence for $(X,D)$ is degenerate. In particular, if $X$ has finite height then the log HdR spectral sequence for $(X,D)$ is degenerate.
        % \item If $r\geq\sigma_0$, then the logarithmic Hodge--de Rham spectral sequence for $(X,D)$ is nondegenerate if and only if $\dim_{\bF_p}(\T_D\cap\N_0)\geq\sigma_0$, where $\T_D\subset\Pic(X)\otimes_{\bZ}\bF_p$ is the $\bF_p$-subspace generated by the classes of the $\ms O(D_i)$ modulo $p$. Moreover, for any supersingular $X$ and any $r\geq\sigma_0$ there exists an snc divisor with exactly $r$ irreducible components having this property.
    \end{enumerate}
\end{theorem}

In fact, on any K3 surface the only possibly nontrivial differential in the log Hodge--de Rham spectral sequence is the differential
\begin{equation}\label{eq:interesting diff, preset}
    \rd:\H^0(X,\Omega^1_X(\log D))\to\H^0(X,\Omega^2_X(\log D))
\end{equation}
on the first page. So, the nondegeneracy of the log HdR spectral sequence amounts to the nontriviality of this particular differential. If this is the case then it turns out that the image of this differential is the one dimensional subspace $\H^0(X,\Omega^2_X)\subset\H^0(X,\Omega^2_X(\log D))$.
% \[
% 	0=\F^3\subset\F^2\subset\F^1\subset\F^0=\H^2_{\dR}(X)
% \]
% on $\H^2_{\dR}$.

We prove \ref{thm:main thm 1} by showing that~\eqref{eq:interesting diff, preset} is nontrivial if and only if the line $\F^2\subset\H^2_{\dR}(X)$ is contained in the $k$-span of the first de Rham Chern classes of the line bundles $\ms O(D_1),\ldots,\ms O(D_r)$, where $\F^2$ is the second piece of the Hodge filtration. This is possible in the supersingular case because of the unusual phenomenon of the non-transversality of the Hodge filtration on $\H^2_{\dR}$ and the subspace generated by algebraic cycles. Let $X$ be a K3 surface and let $\P\subset\H^2_{\dR}(X)$ be the $k$-vector space generated by the image of the de Rham Chern character $c_1:\Pic(X)\to\H^2_{\dR}(X)$. It is always the case that $c_1$ factors through the first piece of the Hodge filtration $\F^1$, so $\P\subset\F^1$. Over the complex numbers, $c_1^{}$ also factors through the first piece of the conjugate filtration, which is a complex vector space having trivial intersection with $\F^2$. Thus $\mathrm{P}$ also has trivial intersection with $\F^2$. In positive characteristic, if $X$ has finite height one can argue similarly using the Newton--Hodge decomposition on the second crystalline cohomology of $X$, which yields the same consequence (see Proposition \ref{prop:easy cases}). However, if $X$ is supersingular, it turns out that $\F^2$ is always contained in $\P$, and in this case the nontriviality of~\eqref{eq:interesting diff, preset} depends on the position in $\P$ of the subspace generated by the $c_1(\ms O(D_i))$. We analyze this situation using some results of Ogus, and this leads to the conditions detailed in the theorem.

As a result of Theorem \ref{thm:main thm 1} we witness the following unusual behaviors of symplectic forms on supersingular K3 surfaces.

\begin{theorem}\label{thm:Liouville form}
    If $X$ is a supersingular K3 surface and $\omega\in\H^0(X,\Omega^2_X)$ is a symplectic form, then there exists an snc divisor $D\subset X$ and a meromorphic 1-form $\eta$ with log poles along $D$ such that $\rd\eta=\omega$. Furthermore, any such $D$ must have $\geq\sigma_0$ irreducible components.
\end{theorem}

\begin{theorem}\label{thm:symplectic form theorem}
    If $X$ is a supersingular K3 surface and $\omega\in\H^0(X,\Omega^2_X)$ is a symplectic form, then there exists a Zariski open cover $\pi:U\to X$ such that $\pi^*\omega$ is exact.
\end{theorem}

% \begin{theorem}\label{thm:de rham coho must die}
%     If $X$ is a supersingular K3 surface then there exists a Zariski open cover $\pi:U\to X$ such that the pullback map $\pi^*:\H^2_{\dR}(X)\to\H^2_{\dR}(U)$ kills all of $\H^2_{\dR}(X)$.
% \end{theorem}

Neither of these can happen in characteristic 0. In fact, in characteristic 0 there cannot exist any nonempty open subset of $X$ over which $\omega$ is exact (see Remark \ref{rem:exact 2 form}). Generalizing the above, we show that any supersingular K3 surface has a distinguished collection of divisors with some remarkable properties.

\begin{theorem}\label{thm:main thm 2}
    Let $X$ be a supersingular K3 surface of Artin invariant $\sigma_0$, let $D\subset X$ be an snc divisor, and let $\omega\in\H^0(X,\Omega^2_X)$ be a symplectic form. The following are equivalent (and moreover on any supersingular $X$ there exist divisors with these properties).
    \begin{enumerate}
        %\item The versal deformation space of the pair $(X,\ms O_X(D))$ is not smooth over $\W$.
        %\item The versal deformation space of the pair $(X,D)$ is not smooth over $\W$.
        \item\label{item:sc1} $\dim_{\bF_p}(\N_D\cap\N_0)\geq\sigma_0$.
        \item\label{item:sc2} The line $\F^2\subset\H^2_{\dR}(X)$ is contained in the $k$-span of the de Rham Chern classes $c_1(\ms O(D_i))$.
        \item\label{item:sc3} The $k$-subspace of $\H^1(X,\Omega^1_X)$ generated by the Hodge Chern classes $c_1^{\H}(\ms O(D_i))$ has dimension $<\dim_{\bF_p}(\N_D)$.
        \item\label{item:sc4} The log HdR spectral sequence for $(X,D)$ is nondegenerate.
        \item\label{item:sc5} There exists a meromorphic 1-form $\eta$ with log poles along $D$ such that $\rd\eta=\omega$.
        %\item\label{item:sc6} The restriction of $\omega$ to the open set $U=X\setminus D$ is exact.
        \item\label{item:sc7} The pair $(X,D)$ does not lift over $W_2$.
    \end{enumerate}
\end{theorem}

\begin{remark}
    Note that property~\eqref{item:sc1} of Theorem \ref{thm:main thm 2} can only hold for a divisor with $\geq\sigma_0$ irreducible components.
\end{remark}

\begin{example}
    Suppose that $X$ is supersingular of Artin invariant $\sigma_0=1$ and $D$ is a smooth irreducible curve whose class in $\Pic(X)$ is not divisible by $p$. Then condition~\eqref{item:sc1} is equivalent to the statement that $D.C$ is divisible by $p$ for all curves $C$ on $X$, condition~\eqref{item:sc2} is equivalent to the de Rham Chern class $c_1(\ms O(D))$ being contained in $\F^2$, and condition~\eqref{item:sc3} is equivalent to the vanishing of the Hodge Chern class $c^{\H}_1(\ms O(D))$.
\end{example}

 The existence parts of the above results are simple consequences of properties of the Picard lattice of a supersingular K3 surface and the existence of linear systems on surfaces with large degrees. It would be interesting to construct such curves explicitly in examples, say on a supersingular Kummer surface or on the Fermat quartic when $p\equiv 3\pmod{4}$. It would also be interesting to find some explicit descriptions of meromorphic symplectic potentials on supersingular K3 surfaces as in \ref{thm:Liouville form}.

\begin{remark}
    The properties of \ref{thm:main thm 2} are not invariant under linear equivalence of divisors. To see this, say $D=\sum_{i=1}^rD_i$ is an snc divisor for which~\eqref{item:sc2} holds. Let $A$ be a smooth irreducible divisor such that $D+A$ is snc and very ample and $c_1(\ms O(A))$ is not in the span of the $c_1(\ms O(D_i))$. Then $D+A$ also satisfies~\eqref{item:sc2}, because $\F^2$ is contained in the span of the Chern classes $c_1(\ms O(D_1)),\ldots,c_1(\ms O(D_r)),c_1(\ms O(A))$. But, if $B$ is a general divisor in the linear system $|D+A|$, then $B$ will be smooth and irreducible and $\F^2$ will not be in the span of $c_1(\ms O(B))$. What is true is that the properties of \ref{thm:main thm 2} are invariant under replacing an irreducible component of $D$ with an irreducible and linearly equivalent divisor. The equivalence of~\eqref{item:sc7} with the rest shows that the liftability of $(X,D)$ over $W_2$ is also not invariant under linear equivalence of divisors.
    %For example, say $\sigma_0=2$, and let $D=D_1+D_2$ be an snc divisor having the properties in the theorem. If $D$ is sufficiently mobile then $D$ will be linearly equivalent to a smooth and irreducible divisor, and when $\sigma_0=2$ property~\eqref{item:sc1} can only hold for a divisor with at least $2$ irreducible components. Even when $\sigma_0=1$ there is a problem - say $D$ is smooth and irreducible and $c_1(\ms O(D))$ spans $\F^2$. Choose a smooth irreducible divisor $A$ such that $D+A$ is snc and $\langle c_1(\ms O(D)),c_1(\ms O(A))\rangle$ is two dimensional. Then $\F^2\subset\langle c_1(\ms O(D)),c_1(\ms O(A))\rangle$. If $A$ is sufficiently ample, then we can find a smooth $B$ such that $B\sim D+A$ 
\end{remark}

% \begin{remark}
%     By the logarithmic version of Deligne--Illusie, if $D$ is an snc divisor on a smooth projective variety $X$ and the pair $(X,D)$ lifts over $W_2$ then the log HdR spectral sequence is degenerate in degrees $<p$ \cite[2.5]{MR894379}.
%     %that is, if $a+b<p$ then $\mathrm{E}_1^{a,b}=\mathrm{E}_{\infty}^{a,b}$) 
%     Thus, the pairs $(X,D)$ in Theorem \ref{thm:main thm 1} cannot lift over $W_2$ (although of course $X$ itself does). In fact the converse is true as well.
% \end{remark}

% \begin{remark}
%     Theorem \ref{thm:main thm 1} can be phrased uniformly in the height as saying that the logarithmic Hodge--de Rham spectral sequence is degenerate if and only if $\dim_{\bF_p}(\T_D\cap\N_0)<\sigma_0$.
% \end{remark}

% \begin{remark}
%     There is a similar story for curves on a supersingular abelian surface.
% \end{remark}

\subsection{Summary of this paper.} In \S\ref{sec:logarithmic de rham} we recall some background on logarithmic differential forms and prove some general results concerning the log HdR spectral sequence. In \S\ref{sec:logarithmic differential forms on K3 surfaces} we consider the log HdR spectral sequence for divisors on a K3 surface. We give the proofs of the above results at the end of \S\ref{sec:logarithmic differential forms on K3 surfaces}.

% We also show that the differential
% \[
%     \H^0(X,\Omega^1_X/B^{r}_X)\to\H^0(X,\Omega^2_X)
% \]
% is nonzero for $r\geq\sigma_0$.

% https://arxiv.org/pdf/2108.03768.pdf (Thm 2.11) requires liftability to $W_2$. 
%Also, see theorem 2.4 of https://arxiv.org/pdf/2008.07700.pdf.
%     Also, our pair should satisfy $\H^2(X,T_X(-\log D))\neq 0$ (see Thm 2.10 of loc. cit.) 
% \end{remark}

\subsection{Acknowledgments}

The author thanks Ambrosi Emiliano for pointing out an error in an earlier version of this paper. The author was supported by NSF RTG Grant \#1840190.

\section{Logarithmic de Rham complexes}\label{sec:logarithmic de rham}
We recall the basic theory of sheaves of logarithmic differential forms and logarithmic de Rham complexes. We take Esnault--Viehweg as our reference \cite[\S2]{MR1193913}.
%References are (Esnault--Viewheg), (Deligne), (Katz), and (Saito).
%"Equations Diff6rentielles & Points Singuliers Reguliers" -Deligne
%"THE REGULARITY THEOREM IN ALGEBRAIC GEOMETRY" by NICHOLAS M. KATZ
%"Theory of logarithmic differential forms and logarithmic vector fields" - Kyoji Saito
%The Saito reference goes beyond the snc divisor case, by the way
Let $k$ be an algebraically closed field, let $X$ be a smooth variety over $k$, and let $D\subset X$ be an snc divisor with irreducible components $D_1,\ldots,D_r$. Set $U=X\setminus D$ and let $j:U\hookrightarrow X$ be the inclusion. The \emph{sheaf of meromorphic 1-forms with logarithmic poles along $D$} is the subsheaf
\[
    \Omega^1_X(\log D)\subset j_*\Omega^1_U
\]
generated as an $\ms O_X$-module by $\Omega^1_X$ and by local sections of the form $\rd f_i/f_i$ ($i=1,\ldots,r$) where $f_i$ is a locally defined equation for $D_i$. We define $\Omega^a_X(\log D)\subset j_*\Omega^a_U$ to be the subsheaf generated as an $\ms O_X$-module by sections $\omega_1\wedge\cdots\wedge\omega_a$ for $\omega_1,\ldots,\omega_a\in\Omega^1_X(\log D)$. Equivalently, $\Omega^a_X(\log D)$ is the subsheaf of $\Omega^a_X(D)$ consisting of forms $\omega$ such that $\rd\omega\in\Omega^{a+1}_X(D)$. This characterization shows that the usual de Rham differentials restrict to maps
\[
    \rd:\Omega^a_X(\log D)\to\Omega^{a+1}_X(\log D).
\]
The \emph{logarithmic de Rham complex} associated to $(X,D)$ is the resulting complex
\[
    \Omega^{\bullet}_X(\log D)=\left[\ms O_X\xrightarrow{\rd}\Omega^1_X(\log D)\xrightarrow{\rd}\Omega^2_X(\log D)\xrightarrow{\rd}\ldots\right].
\]
The \emph{logarithmic Hodge--de Rham spectral sequence} associated to $(X,D)$ is the spectral sequence
    \begin{equation}\label{eq:log HdR ss}
        \mathrm{E}_1^{a,b}=\H^b(X,\Omega^a_X(\log D))\implies\H^{a+b}(X,\Omega^{\bullet}_X(\log D))
    \end{equation}
coming from the naive filtration on $\Omega^{\bullet}_X(\log D)$. Of similar importance are the twists $\Omega^a_X(\log D)(-D)$. These are also stable under $\rd$, and so form a complex
\[
    \Omega^{\bullet}_X(\log D)(-D)=\left[\ms O_X(-D)\xrightarrow{\rd}\Omega^1_X(\log D)(-D)\xrightarrow{\rd}\Omega^2_X(\log D)(-D)\xrightarrow{\rd}\ldots\right]
\]
whose naive filtration gives rise to a spectral sequence
    \begin{equation}\label{eq:log HdR ss dual}
         \mathrm{E}_1^{a,b}=\H^b(X,\Omega^a_X(\log D)(-D))\implies\H^{a+b}(X,\Omega^{\bullet}_X(\log D)(-D)).
    \end{equation}  
We have inclusions $\Omega^{\bullet}_X\subset\Omega^{\bullet}_X(\log D)$ and $\Omega^{\bullet}_X(\log D)(-D)\subset\Omega^{\bullet}_X$ of complexes. The \emph{residue map} is the unique map of $\ms O_X$-modules
\begin{equation}\label{eq:residue map}
    \res:\Omega^1_X(\log D)\to\bigoplus_{i=1}^r\ms O_{D_i}
\end{equation}
that kills $\Omega^1_X$ and sends a local section $\rd f_i/f_i$, where $f_i$ is an equation for $D_i$, to the unit $1_{D_i}\in\ms O_{D_i}$. We have a short exact sequence \cite[\S2, Proposition 2.3]{MR1193913}
\begin{equation}\label{eq:residue SES}
    0\to\Omega^1_X\to\Omega^1_X(\log D)\xrightarrow{\res}\textstyle\bigoplus_{i}\ms O_{D_i}\to 0.
\end{equation}
\begin{pg}[Serre duality for sheaves of logarithmic differentials]

If $X$ is smooth and proper of dimension $d$ then $\Omega^d_X(\log D)=\Omega^d_X(D)$. It follows that the wedge product defines a perfect pairing of $\ms O_X$-modules
\[
    \Omega^a_X(\log D)\otimes\Omega^{d-a}(\log D)(-D)\to\Omega^d_X,
\]
and so we have isomorphisms
\[
    \H^b(X,\Omega^a_X(\log D))\simeq\H^{d-b}(X,\Omega^{d-b}_X(\log D)(-D))^{\vee}.
\]
\end{pg}
    
% \begin{remark}
% In fact, this complex and associated spectral sequence are only one of a family. Given divisors $A,B\subset X$ such that $A+B$ is strict snc, we put
% \[
%     \Omega^a_X(A,B):=\Omega^a_X\log(A+B)(-B).
% \]
% The differential restricts to maps $\Omega^a_X(A,B)\to\Omega^{a+1}_X(A,B)$. Thus we obtain a complex
% \[
%     \Omega^{\bullet}_X(A,B)=\left[\ms O_X(-B)\xrightarrow{\rd}\Omega^1_X(A,B)\xrightarrow{\rd}\Omega^2_X(A,B)\xrightarrow{\rd}\ldots\right]
% \]
% and an associated Hodge--de Rham spectral sequence
% \[
%     \mathrm{E}_1^{a,b}=\H^b(X,\Omega^a_X(A,B))\implies\H^{a+b}(X,\Omega^{\bullet}_X(A,B)).
% \]
% By (REF) there are similar degeneration criterion for these spectral sequences.

% If $X$ is proper of dimension $n$ then the wedge product defines a perfect pairing
% \[
%     \Omega^a_X(A,B)\otimes\Omega^{n-a}(B,A)\to\Omega^n_X
% \]
% of sheaves, and so we have isomorphisms
% \[
%     \H^b(X,\Omega^a_X(A,B))\simeq\H^{n-b}(X,\Omega^{n-b}_X(B,A))^{\ast}.
% \]
    
% \end{remark}
% In higher degrees there is an exact sequence
% \[
%     0\to\Omega^a_X\to\Omega^a_X(\log D)\xrightarrow{\res}\bigoplus_{i=1}^r\Omega^{a-1}_{D_i}(\log( (D-D_i)|_{D_i})).
% \]
% 

\begin{pg}
    We will compute the differential
    \begin{equation}\label{eq:log differential}
        \rd:\H^0(X,\Omega^1_X(\log D))\to\H^0(X,\Omega^2_X(\log D))
    \end{equation}
    in terms of the de Rham cohomology of $X$ and the Chern classes of the $\ms O(D_i)$. Let $\ms F_D\subset\Omega^1_X(\log D)$ denote the subsheaf consisting of sections that may be locally written as $\tau+\lambda\frac{\rd f}{f}$, where $\tau\in\Omega^1_X$, $\lambda\in k$, and $f$ is a local equation for one of the $D_i$. Note that $\ms F_D$ is not an $\ms O_X$-submodule, but is a $k$-submodule.
    The residue map restricts to a map $\ms F_D\to\bigoplus_{i}\underline{k}_{D_i}$, where $\underline{k}_{D_i}\subset\ms O_{D_i}$ is the subsheaf $k\cdot 1_{D_i}$. This map is surjective, as each unit section $1_{D_i}$ may locally be realized as the residue of $\frac{\rd f_i}{f_i}$ where $f_i$ is a local equation for $D_i$. Thus we have a short exact sequence
    \[
        0\to\Omega^1_X\to\ms F_D\xrightarrow{\res}\textstyle\bigoplus_{i}\underline{k}_{D_i}\to 0.
    \]
    Note that, unlike $\Omega^1_X(\log D)$, the sheaf $\ms F_D$ has the property that $\rd$ of any section is regular (sections of the form $\lambda\frac{\rd f}{f}$ with $\lambda\in k$ are closed). Thus the differential restricts to a map $\rd:\ms F_D\to\Omega^2_X$.
\end{pg}
    
     \begin{lemma}\label{lem:it factors}
         If $X$ is smooth and proper over $k$, then the inclusion $\ms F_D\subset\Omega^1_X(\log D)$ induces an isomorphism on global sections, and the differential~\eqref{eq:log differential} factors through $\H^0(X,\Omega^2_X)$ (and even through $\H^0(X,\Z\Omega^2_X)$).
    \end{lemma}
    \begin{proof}
        Consider the diagram
        \[
            \begin{tikzcd}
                0\arrow{r}&\Omega^1_X\arrow[equals]{d}\arrow{r}&\ms F_D\arrow[hook]{d}\arrow{r}{\res}&\bigoplus_{i}\underline{k}_{D_i}\arrow[hook]{d}\arrow{r}&0\\
                0\arrow{r}&\Omega^1_X\arrow{r}&\Omega^1_X(\log D)\arrow{r}{\res}&\bigoplus_{i}\ms O_{D_i}\arrow{r}&0
            \end{tikzcd}
        \]
        with exact rows. The first claim follows by taking cohomology and noting that because the $D_i$ are proper the right vertical inclusion induces an isomorphism on $\H^0$. The second claim follows from the diagram
        \[
            \begin{tikzcd}
                \H^0(X,\ms F_D)\arrow[d, sloped, "\sim"]\arrow{r}{\rd}&\H^0(X,\Omega^2_X)\arrow[hook]{d}\\
                \H^0(X,\Omega^1_X(\log D))\arrow{r}{\rd}&\H^0(X,\Omega^2_X(\log D)).
            \end{tikzcd}
        \]
    \end{proof}

    \begin{remark}
        If $X$ is proper of dimension $d$, then the Serre dual statement to Lemma \ref{lem:it factors} is that the differential
        \[
            \rd:\H^d(X,\Omega^{d-2}_X(\log D)(-D))\to\H^d(X,\Omega^{d-1}_X(\log D)(-D))
        \]
        factors through the surjection $\H^d(X,\Omega^{d-2}_X(\log D)(-D))\twoheadrightarrow\H^d(X,\Omega^{d-2}_X)$.
    \end{remark}
    
    Set $\Omega^{[1,2)}_X=[\Omega^1_X\xrightarrow{\rd}\Z\Omega^2_X]$ and $\C=[\ms F_D\xrightarrow{\rd}\Z\Omega^2_X]$, both with terms in degrees 1 and 2. The diagram
    \begin{equation}\label{eq:SES of complexes}
        \begin{tikzcd}
            0\arrow{r}&\Omega^1_X\arrow{d}{-\rd}\arrow{r}&\ms F_D\arrow{d}{-\rd}\arrow{r}{\res}&\bigoplus_{i}\underline{k}_{D_i}\arrow{r}&0\\
            &\Z\Omega^2_X\arrow[equals]{r}&\Z\Omega^2_X
        \end{tikzcd}
    \end{equation}
    describes a short exact sequence of complexes
    \[
        0\to\Omega^{[1,2)}_X[1]\to\C[1]\to\textstyle\bigoplus_{i}\underline{k}_{D_i}\to 0
    \]
    (here we are following the sign conventions of The Stacks Project \cite[0FNG]{stacks-project}, so $\Omega^{[1,2)}[1]$ and $\C[1]$ have differential $-\rd$). 
    We write $\delta:\bigoplus_{i}\underline{k}_{D_i}\to\Omega^{[1,2)}_X[2]$ for the corresponding morphism in the derived category. We have a commutative diagram
        \begin{equation}\label{eq:diagram in derived cat}
            \begin{tikzcd}
                \C[1]\arrow{r}\arrow[equals]{d}&\ms F_D\arrow{d}[swap]{\res}\arrow{r}{-\rd}&\Z\Omega^2_X\arrow{d}\arrow{r}&\C[2]\arrow[equals]{d}\\
                \C[1]\arrow{r}&\bigoplus_{i}\underline{k}_{D_i}\arrow{r}{\delta}&\Omega^{[1,2)}_X[2]\arrow{r}&\C[2]
            \end{tikzcd}
        \end{equation}
        in the derived category in which the rows are distinguished triangles. 
    
    Consider the map of sheaves 
    \[
        \mathrm{dlog}:\bG_m\to\Omega^1_X,\quad g\mapsto\rd g/g.
    \]
    The \emph{Hodge Chern character} $c_1^{\H}$ is the induced map on $\H^1$. The $\mathrm{dlog}$ map factors through the subsheaf of closed forms, and so defines maps of complexes $\bG_m[-1]\to\Omega^{[1,2)}_X$ and $\bG_m[-1]\to\Omega^{\bullet}_X$. We refer to either of the induced maps on $\H^2$ as the \emph{de Rham Chern character}, denoted $c_1$, and hope that this will not cause any confusion. These maps are compatible via the diagram
    \[
        \begin{tikzcd}
            &\Pic(X)\arrow{dl}[swap]{c_1}\arrow{d}{c_1}\arrow{dr}{c_1^{\H}}&\\
            \H^2_{\dR}(X)&\H^2(X,\Omega^{[1,2)}_X)\arrow{l}\arrow{r}&\H^1(X,\Omega^1_X).
        \end{tikzcd}
    \]

    % The \emph{first Hodge Chern class} is the induced map on $\H^1$
    % \[
    %     c_1^{\H}:\Pic(X)\to\H^1(X,\Omega^1_X).
    % \]
    % The $\rd\log$ map factors through the subsheaf of closed forms, and so we get maps of complexes $\bG_m[-1]\to\Omega^{[1,2)}_X\subset\Omega^{\bullet}_X$.
    % We write $c_1$ for either of the induced maps on $\H^2$. 
    % These are compatible via the diagram
    % \[
    % \begin{tikzcd}[column sep=small]
    %     \Pic(X)\arrow{r}[swap]{c_1}\arrow[bend left=25]{rr}{c_1}&\H^2(X,\Omega^{[1,2)}_X)\arrow{r}&\H^2_{\dR}(X)
    % \end{tikzcd}
    % \]
    % so this will hopefully not cause any confusion. The curved arrow is the \emph{first de Rham Chern character}.

    \begin{lemma}\label{lem:delta}
        The map
        \[
            \delta:k^{\oplus r}\to\H^2(X,\Omega^{[1,2)}_X)    
        \]
        obtained by applying $\H^0$ to $\delta$ is equal to the map $c_1\otimes k$ defined by
        \[
            (\lambda_1,\ldots,\lambda_r)\mapsto\sum_{i=1}^r\lambda_ic_1^{}(\ms O(D_i)).
        \]
    \end{lemma}
    \begin{proof}
        We give a categorified proof. We view $\H^2(X,\Omega^{[1,2)}_X)$ as classifying torsors for the 2-term complex $\Omega^{[1,2)}_X[1]$ and the de Rham Chern class $c_1^{}(\ms O(D_i))\in\H^2(X,\Omega^{[1,2)})$ as the class of the torsor of integrable connections on $\ms O(D_i)$ (as in \cite{MR1237825}). It will suffice to check agreement of the two maps on the basis vectors $1_{D_i}$. The standard description of the boundary map $\H^0\to\H^1$ for the sequence~\eqref{eq:SES of complexes} says that the image of $1_{D_i}$ under $\delta$ is the class of the torsor of 1-forms with log poles along $D_i$, say $\eta$, satisfying $\res(\eta)=1_{D_i}$ and $-\rd\eta=0$. Given such a form we define an integrable connection on $\ms O(D_i)$ as follows. Note first that the quotient rule implies that the usual differential (of rational functions) restricts to a map $\rd:\ms O(D_i)\to\Omega^1_X(2D_i)$. We claim that the sum
        \[
            \rd+\eta:\ms O(D_i)\to\Omega^1_X(2D_i)
        \]
        factors through $\Omega^1_X(D_i)$. We may check this locally. Because $\res(\eta)=1_{D_i}$ we may locally write $\eta=\tau+\frac{\rd f}{f}$ for some 1-form $\tau$ and some local equation $f$ for $D_i$. We may write an arbitrary section of $\ms O(D_i)$ locally as $g/f$ for a regular function $g$, and then compute
        \[
            (\rd+\eta)\left(\frac{g}{f}\right)=\frac{\rd g}{f}-\frac{g\rd f}{f^2}+\left(\tau+\frac{\rd f}{f}\right)\frac{g}{f}=\frac{\rd g}{f}+\tau\frac{g}{f}.
        \]
        This proves the claim. We write
        \[
            \nabla_{\eta}:=\rd+\eta:\ms O(D_i)\to\Omega^1_X(D_i)
        \]
        for the resulting map. It is immediate that $\nabla_{\eta}$ is a connection with curvature $\rd\eta=0$. This defines an isomorphism of torsors, and it follows that $\delta(1_{D_i})=c_1^{}(D_i)$.
    \end{proof}

    \begin{proposition}\label{prop:a key commuting square}
        If $X$ is smooth and proper over $k$, then we have a commutative diagram
        \begin{equation}\label{eq:a diagram that commutes}
            \begin{tikzcd}
                \H^0(X,\Omega^1_X(\log D))\arrow{d}[swap]{\res}\arrow{r}{-\rd}&\H^0(X,\Z\Omega^2_X)\arrow{d}\\
                k^{\oplus r}\arrow{r}{c_1^{}\otimes k}&\H^2(X,\Omega^{[1,2)}_X).
            \end{tikzcd}
        \end{equation}
        If $\H^0(X,\Z\Omega^1_X)=0$ then this diagram is Cartesian.
    \end{proposition}
    \begin{proof}
    If you take cohomology of the diagram~\eqref{eq:diagram in derived cat}, use the identification $\H^0(X,\ms F_D)=\H^0(X,\Omega^1_X(\log D))$, apply Lemma \ref{lem:delta}, and note that $\H^1(X,\Omega^{[1,2)}_X)=\H^0(X,\Z\Omega^1_X)$, you'll get this:
        \[
            \begin{tikzcd}
                0\arrow{r}\arrow{d}&\H^1(X,\C)\arrow[equals]{d}\arrow{r}&\H^0(X,\Omega^1_X(\log D))\arrow{d}[swap]{\res}\arrow{r}{-\rd}&\H^0(X,\Z\Omega^2_X)\arrow{d}\arrow{r}&\H^2(X,\C)\arrow[equals]{d}\\
                \H^0(X,\Z\Omega^1_X)\arrow{r}&\H^1(X,\C)\arrow{r}&k^{\oplus r}\arrow{r}{c_1^{}\otimes k}&\H^2(X,\Omega^{[1,2)}_X)\arrow{r}&\H^2(X,\C).
            \end{tikzcd}
        \]
       If $\H^0(X,\Z\Omega^1_X)=0$ then chasing this diagram yields the claimed Cartesian property.
    \end{proof}

    \begin{remark}
        To complement the above, here is a categorified proof that shows simultaneously the commutativity and Cartesian property of the square~\eqref{eq:a diagram that commutes}. Let $\sTors_X$ denote the stack of torsors for the 2-term complex $\Omega^{[1,2)}_X[1]$. Consider the map
        \[
            \mathscr{c}_1:\left\{1_{D_1},\ldots,1_{D_r}\right\}\to\sTors_X
        \]
        that sends $1_{D_i}$ to the torsor of integrable connections on $\ms O(D_i)$. Using the vector stack structure on $\sTors_X$ this extends uniquely to a map
        \[
            \mathscr{c}_1\otimes k:\bigoplus_{i=1}^r\underline{k}_{D_i}\to\sTors_X.
        \]
        Explicitly, this map sends a tuple $(\lambda_1,\ldots,\lambda_r)$ to the torsor of splittings of the Picard algebroid obtained by taking the Baer sum of the pushouts along $\lambda_i$ of the Atiyah algebroids of the $\ms O(D_i)$. The square~\eqref{eq:a diagram that commutes} is obtained by taking isomorphism classes of global objects in the diagram of stacks
        \begin{equation}\label{eq:a diagram that commutes again}
            \begin{tikzcd}
                \ms F_D\arrow{r}{-\rd}\arrow{d}[swap]{\res}&\Z\Omega^2_X\arrow{d}{\gamma}\\
                \bigoplus_{i}\underline{k}_{D_i}\arrow{r}{\mathscr{c}_1\otimes k}&\sTors_X,
            \end{tikzcd}
        \end{equation}
        where $\gamma$ sends a closed 2-form $\omega$ to the torsor of connections on $\ms O_X$ with curvature $\omega$. We will show that the square~\eqref{eq:a diagram that commutes again} is 2-commutative and homotopy Cartesian (with no vanishing assumptions). This will show that~\eqref{eq:a diagram that commutes} commutes. Furthermore, the set of isomorphisms between two torsors for $\Omega^{[1,2)}_X[1]$ is either empty or is itself a torsor under $\tau_{\leq 0}(\Omega^{[1,2)}_X[1])=\Z\Omega^1_X$. Thus if $\H^0$ of this group vanishes then the homotopy fiber product of global objects of the bottom right corner agrees with the ordinary fiber product of the sets of isomorphism classes of global objects, and so we will conclude that~\eqref{eq:a diagram that commutes} is Cartesian.

        We first construct a natural isomorphism between the down-right and right-down functors. We describe the action of this transformation directly in the case when $\eta\in\ms F_D$ is a 1-form with log poles such that $\res(\eta)=1_{D_i}$ for some $i$, and obtain the general case from this by linearity. Giving an isomorphism of torsors $\gamma(-\rd\eta)\simeq\mathscr{c}_1(1_{D_i})$ is the same as giving a connection on $\ms O(D_i)$ with curvature $\rd\eta$. As in the proof of Lemma \ref{lem:delta}, $\nabla_{\eta}=\rd+\eta$ gives such a connection, and we use it to define the desired natural isomorphism. It remains to show that the square~\eqref{eq:a diagram that commutes again} is homotopy Cartesian. Again by linearity it will suffice to consider the fiber over $1_{D_i}$. The fiber of $\ms F_D$ over $1_{D_i}$ is $\ms F_{D_i}$. The homotopy fiber $\gamma^{-1}(\mathscr{c}_1(1_{D_i}))$ of $\gamma$ over $\mathscr{c}_1(1_{D_i})$ may be identified with the set of pairs $(\omega,\nabla)$, where $\omega$ is a closed 2-form and $\nabla$ is a connection on $\ms O(D_i)$ with curvature $-\omega$. We will show that the map
        \begin{equation}\label{eq:map of pairs}
            \ms F_{D_i}\to\gamma^{-1}(\mathscr{c}_1(1_{D_i})),\quad\eta\mapsto(-\rd\eta,\nabla_{\eta})
        \end{equation}
        is an isomorphism. Indeed, the left hand side carries a simply transitive action of $\Omega^1_X$ by addition. The right hand side also carries an action of $\Omega^1_X$, where a 1-form $\tau$ acts by $(\omega,\nabla)\mapsto (\omega-\rd\tau,\nabla+\tau)$. This action is also simply transitive, because it projects to a simply transitive action on the torsor of connections on $\ms O(D_i)$. The map~\eqref{eq:map of pairs} is compatible with these actions, and hence is an isomorphism.\qed
    \end{remark}

\begin{proposition}\label{prop:degen criterion for a surface}
	Let $X$ be a smooth projective surface with $\H^1(X,\ms O_X)=0$ and degenerate Hodge--de Rham spectral sequence. If $D\subset X$ is an snc divisor, then every differential on every page of the log Hodge--de Rham spectral sequence for $(X,D)$ vanishes, except possibly the $\mathrm{E}_1$ differential
	\begin{equation}\label{eq:spicy differential}
	       \rd^{1,0}:\H^0(X,\Omega^1_X(\log D))\to\H^0(X,\Omega^2_X(\log D)).
	\end{equation}
\end{proposition}
\begin{proof}
	 For $b=0,1,2$ the differential
	\[
		\rd^{1,b}:\H^b(X,\ms O_X)\to\H^b(X,\Omega^1_X(\log D))
	\]
	on the $\mathrm{E}_1$-page of the log HdR spectral sequence factors through the differential $\H^b(X,\ms O_X)\to\H^b(X,\Omega^1_X)$ on the $\mathrm{E}_1$-page of the ordinary HdR spectral sequence. By assumption the latter vanishes, so the former does as well. The same reasoning shows the vanishing of every differential on any page of the log HdR spectral sequence whose source is $\H^b(X,\ms O_X)$ for some $b$. In particular, all differentials on the $\mathrm{E}_s$-page vanish for $s\geq 2$. On the $\mathrm{E}_1$-page the remaining differentials to consider are
	\[
		\rd^{1,2}:\H^2(X,\Omega^1_X(\log D))\to\H^2(X,\Omega^2_X(\log D))
	\]
	and
	\[
		\rd^{1,1}:\H^1(X,\Omega^1_X(\log D))\to\H^1(X,\Omega^2_X(\log D)).
	\]
	By Serre duality we have $\H^2(X,\Omega^2_X(\log D))\simeq\H^0(X,\ms O(-D))^{\vee}$ and as $D$ is effective this vanishes, so $\rd^{1,2}$ is trivial. To deal with $\rd^{1,1}$, let $\Q$ denote the cokernel of the inclusion $\Omega^1_X(\log D)(-D)\subset\Omega^1_X$. Taking cohomology of the diagram
	\[
		\begin{tikzcd}
			0\arrow{r}&\ms O(-D)\arrow{d}{\rd}\arrow{r}&\ms O_X\arrow{d}{\rd}\arrow{r}&\ms O_D\arrow{d}{\overline{\rd}}\arrow{r}&0\\
			0\arrow{r}&\Omega^1_X(\log D)(-D)\arrow{r}&\Omega^1_X\arrow{r}&\Q\arrow{r}&0
		\end{tikzcd}
	\]
	and using $\H^1(X,\ms O_X)=0$ we get
	\[
		\begin{tikzcd}
			\H^0(D,\ms O_D)\arrow{d}{\rd^{\prime\prime}}\arrow{r}&\H^1(X,\ms O(-D))\arrow{d}{\rd'}\arrow{r}&0\arrow{d}\\
			\H^0(X,\Q)\arrow{r}&\H^1(X,\Omega^1_X(\log D)(-D))\arrow{r}&\H^1(X,\Omega^1_X).
		\end{tikzcd}
	\]
	Let $D_1,\ldots,D_r$ be the irreducible components of $D$. Using \cite[2.3]{MR1193913} we have that the restriction maps give an embedding of $\Q$ as a subsheaf of $\bigoplus_{i}\Omega^1_{D_i}$. This embedding and the map $\overline{\rd}$ fit into the diagram
	\[
		\begin{tikzcd}
			\ms O_D\arrow[hook]{r}\arrow{d}{\overline{\rd}}&\bigoplus_{i}\ms O_{D_i}\arrow{d}{\bigoplus_i\rd}\\
			\Q\arrow[hook]{r}&\bigoplus_{i}\Omega^1_{D_i}.
		\end{tikzcd}
	\]
	By the degeneration of the ordinary HdR spectral sequences for the $D_i$ each of the differentials $\rd:\ms O_{D_i}\to\Omega^1_{D_i}$ is trivial on $\H^0$. It follows that $\rd^{\prime\prime}=0$, hence $\rd'=0$. But $\rd'$ is Serre dual to $\rd^{1,1}$, so $\rd^{1,1}=0$.
\end{proof}

\begin{figure}[H]
	\centering
	\begin{equation*}
	\begin{tikzcd}
	\H^2(X,\ms O_X)\arrow{r}\arrow[dashed]{drr}&\H^2(X,\Omega^1_X(\log D))\arrow{r}&\H^2(X,\Omega^2_X(\log D))    \\
	\H^1(X,\ms O_X)\arrow{r} \arrow[dashed]{drr}&\H^1(X,\Omega^1_X(\log D))\arrow{r}&\H^1(X,\Omega^2_X(\log D)) \\
	\H^0(X,\ms O_X) \arrow{r} &\H^0(X,\Omega^1_X(\log D))\arrow{r}&\H^0(X,\Omega^2_X(\log D))
	\end{tikzcd}
	\end{equation*}
	\caption{The $\mathrm{E}_1$-page of the log Hodge--de Rham spectral sequence for a surface. The dashed arrows indicate the positions of the differentials on the $\mathrm{E}_2$-page.}
	\label{fig:surfaceslope}
\end{figure}

For a smooth proper variety $X$ and an snc divisor $D$ the differentials in the log Hodge--de Rham spectral sequence for $(X,D)$ are Serre dual to those in the spectral sequence~\eqref{eq:log HdR ss dual}. In particular, for a surface the differential~\eqref{eq:spicy differential} in the log Hodge--de Rham spectral sequence is Serre dual to the differential
\begin{equation}\label{eq:interesting differential, in variant ss}
\rd:\H^2(X,\ms O(-D))\to\H^2(X,\Omega^1_X(\log D)(-D)).
\end{equation}
Thus Theorem \ref{thm:degeneration criterion} implies the following.

\begin{corollary}
	If $X$ is a smooth projective surface with $\H^1(X,\ms O_X)=0$ and degenerate Hodge--de Rham spectral sequence, then every differential on every page of the spectral sequence~\eqref{eq:log HdR ss dual} vanishes, except possibly~\eqref{eq:interesting differential, in variant ss}.
\end{corollary}

%\begin{remark}
%	We have shown that if the differential $\rd^{1,0}$ is nonzero, then it has rank 1, with image equal to $\H^0(X,\Omega^2_X)\subset\H^0(X,\Omega^2_X(\log D))$.
%\end{remark}

\begin{remark}
	If $X$ is a surface with degenerate HdR spectral sequence and $\H^1(X,\ms O_X)\neq 0$, it seems possible that the differential
	\[
		\rd^{1,1}:\H^1(X,\Omega^1_X(\log D))\to\H^1(X,\Omega^2_X(\log D))
	\]
	could also be nonzero, but I do not have an example.
\end{remark}

\section{Logarithmic differential forms on K3 surfaces}\label{sec:logarithmic differential forms on K3 surfaces}

Let $X$ be a K3 surface over an algebraically closed field $k$. We write $\H^n_{\dR}(X):=\H^n(X,\Omega^{\bullet}_{X/k})$ for the algebraic de Rham cohomology groups of $X$ over $k$. We consider the Hodge filtration
\[
    0=\F^3\subset\F^2\subset\F^1\subset\F^0=\H^2_{\dR}(X)
\]
on $\H^2_{\dR}$, where
\[
    \F^i:=\mathrm{im}\left(\H^2(X,\Omega^{\geq i}_X)\to\H^2_{\dR}(X)\right).
\]
Due to the degeneration of the Hodge--de Rham spectral sequence we have natural identifications
\[
    \F^0/\F^1=\H^2(X,\ms O_X),\quad\F^1/\F^2=\H^1(X,\Omega^1_X),\quad\F^2/\F^3=\F^2=\H^0(X,\Omega^2_X),
\]
\[
    \text{and}\quad\F^1=\H^2(X,\Omega^{[1,2)}_X).
\]
The vector spaces $\H^2(X,\ms O_X)$, $\H^1(X,\Omega^1_X)$, and $\H^0(X,\Omega^2_X)$ have dimensions $1,20,1$.

\begin{theorem}\label{thm:degeneration criterion}
	Let $X$ be a K3 surface and let $D\subset X$ be an snc divisor with irreducible       components $D_1,\dots,D_r$. The following are equivalent.
	\begin{enumerate}
		\item\label{item:1} The logarithmic Hodge--de Rham spectral sequence for $(X,D)$ is nondegenerate.
		\item\label{item:2} The differential
		\[
		  \rd:\H^0(X,\Omega^1_X(\log D))\to\H^0(X,\Omega^2_X(\log D))
		\]
		on the $\mathrm{E}_1$-page is nonzero.
		\item\label{item:3} The line $\F^2\subset\H^2_{\dR}(X)$ is contained in the $k$-span of the de Rham Chern classes $c_1^{}(\ms O(D_1)),\dots,c_1^{}(\ms O(D_r))\in\H^2_{\dR}(X)$.
	\end{enumerate}
\end{theorem}
\vspace{.1cm}
\begin{proof}
	~\eqref{item:1}$\iff$~\eqref{item:2} follows from Proposition \ref{prop:degen criterion for a surface}. To show~\eqref{item:2}$\iff$~\eqref{item:3} we apply Proposition \ref{prop:a key commuting square}. As $\H^0(X,\Omega^1_X)=0$ we have $\H^0(X,\Z\Omega^1_X)=0$, so the square~\eqref{eq:a diagram that commutes} is Cartesian. The map $\H^2(X,\Omega^{[1,2)}_X)\to\H^2_{\dR}(X)$ is injective, so there is no harm in replacing the lower right entry with $\H^2_{\dR}(X)$, and as $X$ is a surface we have $\Z\Omega^2_X=\Omega^2_X$. We obtain a Cartesian diagram
	\[
	\begin{tikzcd}
	\H^0(X,\Omega^1_X(\log D))\arrow{r}{-\rd}\arrow[hook]{d}[swap]{\res}&\H^0(X,\Omega^2_X)\arrow[hook]{d}\\
	k^{\oplus r}\arrow{r}{c_1^{}\otimes k}&\H^2_{\dR}(X),
	\end{tikzcd}
	\]
	where the injectivity of the vertical arrows follows from the vanishing of $\H^0(X,\Omega^1_X)$. The commutativity of the square shows~\eqref{item:2}$\implies$~\eqref{item:3}. Conversely, if $\omega\in\H^0(X,\Omega^2_X)$ is a generator and $\omega=\sum_{i}\lambda_ic_1^{}(\ms O(D_i))$ for some $\lambda_1,\ldots,\lambda_r\in k$, then by the Cartesian property there exists a meromorphic 1-form $\eta$ with log poles along $D$ such that $\res(\eta)=(\lambda_1,\ldots,\lambda_r)$ and $-\rd\eta=\omega$. This shows~\eqref{item:3}$\implies$~\eqref{item:2} and completes the proof.
\end{proof}

\begin{pg}
Let $\mathrm{P}\subset\H^2_{\dR}(X)$ denote the $k$-vector space generated by the image of the de Rham Chern character $c_1:\Pic(X)\to\H^2_{\dR}(X)$. Equivalently, $\mathrm{P}$ is the image of the induced map of $k$-vector spaces
\[
	c_1^{}\otimes k:\Pic(X)\otimes_{\bZ} k\to\H^2_{\dR}(X).
\]
As $c_1$ factors through $\F^1$, we have $\P\subset\F^1$. Consider classes $L_1,\ldots,L_r\in\Pic(X)$, let $\A\subset\Pic(X)\otimes k$ be the $k$-subspace they generate, and let $\B\subset\H^2_{\dR}(X)$ be the $k$-vector space generated by $c_1(L_1),\ldots,c_1(L_r)$ (that is, the image of $\A$ under $c_1\otimes k$). The situation looks like this:
\begin{equation}\label{eq:the situation}
    \begin{tikzcd}
        &&\F^2\arrow[d,phantom,sloped,"\subset"]\\
        \B\arrow[r,phantom,"\subset"]&\P\arrow[r,phantom,sloped,"\subset"]&\F^1\arrow[d,phantom,sloped,"\subset"]\\
        &&\H^2_{\dR}.
    \end{tikzcd}
\end{equation}
Motivated by \ref{thm:degeneration criterion} we will consider the question of when $\F^2$ is contained in $\B$. We first observe that this cannot happen outside of the supersingular case.

\begin{proposition}\label{prop:easy cases}
    If $k$ has characteristic 0, or if $k$ has positive characteristic and $X$ has finite height, then $\P\cap\F^2=0$.
\end{proposition}
\begin{proof}
    Over the complex numbers Hodge theory gives a direct sum decomposition of complex vector spaces 
    \[
        \H^2_{\dR}(X)=\H^{0,2}\oplus\H^{1,1}\oplus\H^{2,0}.
    \]
    The de Rham Chern character factors through $\H^{1,1}$, so $\P\subset\H^{1,1}$. As $\F^2=\H^{2,0}$ we have $\P\cap\F^2=0$. When $k$ has characteristic $0$ the same consequence holds by the Lefschetz principle. If $k$ has positive characteristic and $X$ has finite height then we can make a similar argument using the Newton--Hodge decomposition
    \[
        \H^2(X/\W)=\H_{<1}\oplus\H_1\oplus\H_{>1}
    \]
    on the second crystalline cohomology of $X$. This is a decomposition into $\F$-crystals having slopes as indicated in the subscripts. The summands on the right are free $\W$-modules, of ranks $h, 22-2h, h$. Because the crystalline cohomology of $X$ is torsion free, the canonical map
    \[
        \H^2(X/\W)\twoheadrightarrow\H^2_{\dR}(X)
    \]
    is surjective and identifies $\H^2_{\dR}$ with the reduction modulo $p$ of $\H^2(X/\W)$. Let $\overline{\H}_{<1}$, $\overline{\H}_1$, and $\overline{\H}_{>1}$ denote the reductions modulo $p$ of $\H_{<1}$, $\H_1$, and $\H_{>1}$. We get a direct sum decomposition of $k$-vector spaces
    \[
        \H^2_{\dR}(X)=\overline{\H}_{<1}\oplus\overline{\H}_1\oplus\overline{\H}_{>1}
    \]
    having dimensions $h,22-2h,h$. It follows from Mazur's theorem relating the Frobenius on $\H^2(X/\W)$ to the Hodge filtration on $\H^2_{\dR}$ that $\F^2\subset\overline{\H}_{>1}$ (see eg. \cite[8.26, 8.30]{BO78}). On the other hand, the crystalline Chern character, which lifts the de Rham Chern character, factors through $\H_1$. Thus the de Rham Chern character factors through $\overline{\H}_1$, and so $\P\subset\overline{\H}_1$. Therefore $\P\cap\F^2=0$.
\end{proof}

\begin{corollary}
     If $k$ has characteristic 0, or if $k$ has positive characteristic and $X$ has finite height, then for any snc divisor $D\subset X$ the logarithmic Hodge--de Rham spectral sequence for $(X,D)$ is degenerate.
\end{corollary}
\begin{proof}
    Let $D_1,\ldots,D_r$ be the irreducible components of $D$. Apply \ref{prop:easy cases} to the classes $L_i=[\ms O(D_i)]$ and use Theorem \ref{thm:degeneration criterion}.
\end{proof}

\begin{remark}
    In the cases of \ref{prop:easy cases} the map $c_1\otimes k$ is injective, hence induces an isomorphism $\Pic(X)\otimes k\simeq\P$, and so $\P\simeq k^{\rho}$ where $\rho$ is the Picard number of $X$.
\end{remark}
\end{pg}

\begin{pg}
The case when $X$ is supersingular is more interesting. In particular, we will see that here the triviality/nontriviality of the intersection $\B\cap\F^2$ depends on $\B$. To analyze the situation we will use a beautiful bit of semilinear algebra worked out by Ogus \cite{Ogus78}.
\end{pg}
Write $\N=\Pic(X)$. We consider the semilinear endomorphism $\varphi$ of the $k$-vector space $\N\otimes k$ given by
    \[
        \varphi:\N\otimes k\to\N\otimes k,\quad\quad v\otimes\lambda\mapsto v\otimes\lambda^p.
    \]
    As $k$ is perfect $\varphi$ is invertible. The \emph{characteristic subspace} associated to $X$ is defined by
    \[
        \K:=\varphi^{-1}(\ker (c_1\otimes k))\subset \N\otimes k.
    \]
    Thus we have an exact sequence of $k$-vector spaces
    \[
        0\to\varphi(\K)\to \N\otimes k\xrightarrow{c_1\otimes k}\H^2_{\dR}(X).
    \]
    We have
    \begin{equation}\label{eq:char subspace 0}
        \dim_k \K=\sigma_0.
    \end{equation}
    This follows from \cite[3.11]{Ogus78}.\footnote{To make contact with Ogus's notation, we note that Ogus's $\W$-lattice denoted $\mathrm{E}$ satisfies $\Pic(X)\otimes_{\bZ}\W=\mathrm{E}$. The identification of our $\K$ with Ogus's follows from the diagram \cite[3.16.2]{Ogus78}.} We have $1\leq\sigma_0\leq 10$, and so in particular $c_1\otimes k$ is not injective, unlike in the finite height case. (This could be seen more directly by noting that $\Pic(X)\otimes k$ is a $k$-vector space of dimension 22, and that $c_1$ factors through $\F^1$, which has dimension only $21$.)
    
    The following facts are consequences of \cite[3.12]{Ogus78} (see also \cite[3.19, 3.20]{Ogus78}). First, the subspace $\K$ is in a special position with respect to the operator $\varphi$; specifically, we have
    \begin{equation}\label{eq:char subspace}
        \dim_k(\K+\varphi(\K))=\sigma_0+1,
    \end{equation}
    and letting $\U:=\sum_{i\geq 0}\varphi^i(\K)\subset\N\otimes k$ denote the smallest $\varphi$-stable subspace of $\N\otimes k$ containing $\K$, we have
    \begin{equation}\label{eq:char subspace 2}
        \dim_k(\U)=2\sigma_0.
    \end{equation}
    Here, we say that a $k$-subspace $\V\subset\N\otimes k$ is \emph{$\varphi$-stable} if $\varphi(\V)\subset\V$ (equivalently, $\varphi(\V)=\V$). This is equivalent to the existence of an $\bF_p$-subspace $\T\subset\N\otimes\bF_p$ such that $\V=\T\otimes_{\bF_p}k$. Second, the subspace $\U$ can be described in a different way using the intersection pairing on $\N$. The discriminant group $\N^{\vee}/\N$ is killed by $p$, so we have the containment $p\N^{\vee}\subset\N$. Explicitly, $p\N^{\vee}$ is the subgroup of classes $L\in\N$ such that $L.M$ is divisible by $p$ for all $M\in\N$. We let
    \[
        \N_0:=\frac{p\N^{\vee}}{p\N}\subset\frac{\N}{p\N}=\N\otimes\bF_p
    \]
    denote its image modulo $p$. Multiplication by $p$ induces an isomorphism $\N^{\vee}/\N\simeq\N_0$, so $\N_0$ is an $\bF_p$-vector space of dimension $2\sigma_0$. Moreover, we have
    \[
    	\U=\N_0\otimes_{\bF_p}k
    \]
    as $k$-subspaces of $\N\otimes k$ \cite[3.16]{Ogus78}. Finally, in contrast to the finite height case, when $X$ is supersingular the second piece of the Hodge filtration is contained in $\P$. In fact, we even have $\F^2\subset\P_0$, where $\P_0$ is the image of $\U$ under $c_1\otimes k$.
    As a consequence of the relation between $\varphi$ and the Cartier operator, it follows that $c_1\otimes k$ maps $\K+\varphi(\K)$ to $\F^2$, and moreover we have a commutative diagram of $k$-vector spaces
  	\begin{equation}\label{eq:diagram c_1 dr again}
  		\begin{tikzcd}
  			0\arrow{r}&\varphi(\K)\arrow[d,phantom,sloped,"="]\arrow{r}&\K+\varphi(\K)\arrow[d,phantom,sloped,"\subset"]\arrow{r}{c_1\otimes k}&\F^2\arrow[d,phantom,sloped,"\subset"]\arrow{r}&0\\
  			0\arrow{r}&\varphi(\K)\arrow[d,phantom,sloped,"="]\arrow{r}&\U\arrow[d,phantom,sloped,"\subset"]\arrow{r}{c_1\otimes k}&\P_0\arrow[d,phantom,sloped,"\subset"]\arrow{r}&0\\
  			0\arrow{r}&\varphi(\K)\arrow{r}&\N\otimes k\arrow{r}{c_1\otimes k}&\P\arrow{r}\arrow[d,phantom,sloped,"\subset"]&0\\
  			&&&\F^1\arrow[d,phantom,sloped,"\subset"]&\\
  			&&&\H^2_{\dR}
  		\end{tikzcd}
  	\end{equation}
    with exact rows \cite[3.20.1]{Ogus78}. In particular, from the top two rows we see that
  	\[
  		\K+\varphi(\K)=(c_1\otimes k)^{-1}(\F^2).
  	\]
    
    So, to understand the intersection $\B\cap\F^2$ we need to understand the relative position of $\A$ with respect to the subspaces $\K+\varphi(\K)$ and $\varphi(\K)$. Actually, and maybe surprisingly, we will see that all that matters is the relative position of $\A$ with respect to $\U$. To explain this it is clarifying to consider more generally the intersection of $\A$ with the terms of a certain complete flag in $\U$ extending the filtration $0\subset \K\subset \K+\varphi(\K)\subset \U$. We define an ascending $\bZ$-indexed filtration 
    \[
        \cdots\subset\U_m\subset\U_{m+1}\subset\cdots\subset\U
    \]
    by
   	\[
   		\U_m=\begin{cases}\varphi^{-\sigma_0+m}(\K)\cap\ldots\cap \K,&\quad\text{if }m\leq\sigma_0,\\
   		\K+\ldots+\varphi^m(\K),&\quad\text{if }m\geq\sigma_0.
   		\end{cases}
   	\]
   	In particular, we have $\U_{\sigma_0}=\K$ and $\U_{\sigma_0+1}=\K+\varphi(\K)$.
  
    \begin{lemma}\label{lem:U lemma}
    	For $0\leq m\leq 2\sigma_0$ we have $\dim \U_m=m$.
    \end{lemma}
    \begin{proof}
        When $m=\sigma_0$ this is~\eqref{eq:char subspace 0}. The claim follows from this using induction combined with the relations~\eqref{eq:char subspace} and~\eqref{eq:char subspace 2} and the fact that $\varphi$ is invertible.
        % Say $m\geq\sigma_0$. As $\varphi$ is invertible we have $\dim_k\varphi^m(\K)=\sigma_0$ and $\dim_k(\varphi^m(\K)+\varphi^{m+1}(\K))=\sigma_0+1$, so $\U_{m+1}/\U_m$ has dimension 0 or 1. Thus each of the inclusions $\U_{\sigma_0}\subset\U_{\sigma_0+1}\subset\ldots$ increases the dimension by at most 1. As $\U$ has dimension $2\sigma_0$ there must be exactly $\sigma_0$ of these inclusions that are strict. If it is the case that $\U_{m+1}/\U_m=0$, then we have $\varphi^{m+1}(\K)\subset\U_m$, so $\U_m=\U_{m+1}=\ldots=\U$. Thus the strict inclusions must all occur first, and this proves the result for $m\geq\sigma_0$. Now say $m\leq\sigma_0$. We note that
        % \[
        %     \frac{\K+\varphi(\K)}{\varphi(\K)}\simeq\frac{\K}{\K\cap\varphi(\K)}
        % \]
        % has dimension 1. Thus $\varphi^{-\sigma_0+m}(\K)\cap\varphi^{-\sigma_0+m+1}(\K)$ has dimension $\sigma_0-1$, so for $m\leq\sigma_0$ the quotient $\U_{m}/\U_{m-1}$ has dimension 0 or 1.
    \end{proof}
The above shows that we have $\U_m=0$ for $m\leq 0$ and $\U_m=\U$ for $m\geq 2\sigma_0$, and the $\U_m$ filtration looks like this:
\begin{equation}\label{eq:U filtration}
	\begin{tikzcd}[column sep=small]
		0\arrow[r,phantom,"="]&\U_0\arrow[r,phantom,"\subset"]&\U_1\arrow[r,phantom,"\subset"]&\cdots\arrow[r,phantom,"\subset"]&\U_{\sigma_0}\arrow[d,phantom,sloped,"="]\arrow[r,phantom,"\subset"]&\U_{\sigma_0+1}\arrow[d,phantom,sloped,"="]\arrow[r,phantom,sloped,"\subset"]&\cdots\arrow[r,phantom,"\subset"]&\U_{2\sigma_0}\arrow[r,phantom,"="]&\U\\
		&&&&\K\arrow[r,phantom,sloped,"\subset"]&\K+\varphi(\K).
	\end{tikzcd}
\end{equation}
Moreover, each term in the top row has codimension 1 in its successor. This implies that the $\U_m$ satisfy the recurrence relations
        \begin{align*}
            \U_{m+1}\,\,&=\U_m+\varphi(\U_m),\quad m>0,\\
            \varphi(\U_{m-1})&=\U_m\cap\varphi(\U_m),\quad m<2\sigma_0.
        \end{align*}
Another way to say all of this is that $\U_1$ has dimension 1, and if $e\in\U_1$ is a generator and $e_m:=\varphi^{m-1}(e)$, then for $1\leq m\leq 2\sigma_0$ the vectors $(e_1,\ldots,e_m)$ are a basis for $\U_m$.

\begin{remark}
    The $\U_m$ can be interpreted more intrinsically in terms of cohomology groups of sheaves of higher closed forms on $X$; see \cite[\S11]{MR1776939}.
\end{remark}

The following shows that the intersection of a $\varphi$-stable subspace of $\U$ with one of the $\U_m$ is as small as possible.

    \begin{lemma}\label{lem:dimension bounds}
        If $\V\subset\U$ is a $\varphi$-stable subspace of dimension $s_0$, then for $0\leq m\leq 2\sigma_0$ we have
        \[
            \dim(\V\cap \U_m)=\begin{cases}
                0,&\quad\text{if }s_0+m<2\sigma_0,\\
               s_0+m-2\sigma_0,&\quad\text{if }s_0+m\geq 2\sigma_0.
            \end{cases}
        \]
        In particular, the dimension of $\V\cap\U_m$ depends only on $m$ and the dimension of $\V$.
    \end{lemma}
    \begin{proof}
        Set $\V_m=\V\cap \U_m$. Consider the intersection of $\V$ with the terms of the filtration~\eqref{eq:U filtration}. It looks like this:
        \begin{equation}\label{eq:V filtration}
            0=\V_0\subset \V_1\subset\cdots\subset \V_{2\sigma_0}=\V.
        \end{equation}
        Each of the inclusions increases dimension by at most $1$, and $\V_{2\sigma_0}$ has dimension $s_0$, so there must be exactly $s_0$ inclusions that increase dimension by $1$, with the remaining $2\sigma_0-s_0$ inclusions being equalities. We claim that the equalities must all occur first. Suppose that $0\leq m\leq 2\sigma_0-1$ is an index such that $\V_m=\V_{m+1}$. As $\V$ is stable under $\varphi$ we have
        \[
            \varphi(\V_m)=\varphi(\V\cap\U_m)=\V\cap\varphi(\U_{m})\subset\V\cap\U_{m+1}=\V_{m+1}.
        \]
        Thus $\varphi(\V_m)\subset\V_m$, so $\varphi(\V_m)=\V_m$. We obtain
        \[
        	\varphi(\V_{m-1})=\varphi(\V\cap\U_{m-1})=\V\cap\varphi(\U_{m-1})=\V\cap(\U_m\cap\varphi(\U_m))=\V_m\cap\varphi(\V_m)=\varphi(\V_m).
        \]
        Therefore $\V_{m-1}=\V_{m}$. Applying this repeatedly we get $\V_{m+1}=\V_m=\V_{m-1}=\ldots=\V_1=\V_0=0$, which proves the claim. We conclude that~\eqref{eq:V filtration} looks like
        \[
            0=\V_0=\V_1=\ldots=\V_{2\sigma_0-s_0}\subsetneq \V_{2\sigma_0-s_0+1}\subsetneq\ldots\subsetneq \V_{2\sigma_0}=\V
        \]
        where each of the nontrivial inclusions increases the dimension by $1$. So, for $m<2\sigma_0-s_0$ we have $\V_m=0$, and for $m\geq 2\sigma_0-s_0$ we have $\dim(\V_m)=m-(2\sigma_0-s_0)$. After rearranging, we see this is what we wanted to prove. 
    \end{proof}

    We can now answer our question. We refresh the notation: $L_1,\ldots,L_r$ are classes 
    in $\Pic(X)=\N$, $\A\subset\N\otimes k$ is the $k$-subspace they generate, and $\B\subset\H^2_{\dR}(X)$ is the $k$-subspace generated by $c_1(L_1),\ldots,c_1(L_r)$. Also, let $\C\subset\H^1(X,\Omega^1_X)$ be the $k$-subspace spanned by the Hodge Chern classes $c_1^{\H}(\ms O(D_i))$. Set $s=\dim(\A)$ and $s_0=\dim(\A\cap\U)$.
    \begin{proposition}\label{prop:sigma condition}
        The dimensions of $\B,\B\cap\F^2,$ and $\C$ are given by
        \[
            \dim(\B)=\begin{cases}
                s,&\text{if }s_0<\sigma_0,\\
                s-s_0+\sigma_0,&\text{if }s_0\geq\sigma_0,
            \end{cases}
            \quad\quad\quad
             \dim(\B\cap\F^2)=\begin{cases}
                0,&\text{if }s_0<\sigma_0,\\
                1,&\text{if }s_0\geq\sigma_0,
            \end{cases}
        \]
        \[
            \text{and}\quad\dim(\C)=\begin{cases}
                s,&\text{if }s_0<\sigma_0,\\
                s-s_0+\sigma_0-1,&\text{if }s_0\geq\sigma_0.
            \end{cases}
        \]
        In particular, $\dim(\B)$ and $\dim(\C)$ depend only on $s$ and $s_0$, and $\dim(\B\cap\F^2)$ depends only on $s_0$.
    \end{proposition}
    \begin{proof}
    Note that $\B=(c_1\otimes k)(\A)$ and $\C=\pi(\B)$, where $\pi:\F^1\to\F^1/\F^2=\H^1(X,\Omega^1_X)$ is the quotient. The situation looks like this:
    \[
        \begin{tikzcd}
            \A\arrow[d,phantom,sloped,"\subset"]\arrow[two heads]{r}&\B\arrow[d,phantom,sloped,"\subset"]\arrow[two heads]{r}&\C\arrow[d,phantom,sloped,"\subset"]\\
            \Pic(X)\otimes k\arrow{r}{c_1\otimes k}&\F^1\arrow{r}{\pi}&\H^1(X,\Omega^1_X).
        \end{tikzcd}
    \]
    The kernel of $c_1\otimes k$ is $\varphi(\K)$, the kernel of $\pi$ is $\F^2$, and the kernel of the composition is $\K+\varphi(\K)$. So, we have $\B\simeq\A/(\A\cap\varphi(\K))$, $\B\cap\F^2=(\A\cap(\K+\varphi(\K)))/(\A\cap\varphi(\K))$, and $\C\simeq\B/(\B\cap\F^2)\simeq\A/(\A\cap(\K+\varphi(\K)))$. Therefore it will suffice to compute the dimensions of $\A\cap\varphi(\K)$ and $\A\cap(\K+\varphi(\K))$. As $\varphi(\K)$ and $\K+\varphi(\K)$ are contained in $\U$, these subspaces are equal to $\V\cap\varphi(\K)$ and $\V\cap(\K+\varphi(\K))$ respectively, where $\V=\A\cap\U$. Note that $\V$ is $\varphi$-stable, so we have $\dim(\V\cap\K)=\dim(\V\cap\varphi(\K))$. We now apply \ref{lem:dimension bounds} to $\V$ with $m=\sigma_0$ and $m=\sigma_0+1$. When $s_0<\sigma_0$ this gives $\V\cap(\K+\varphi(\K))=\V\cap\varphi(\K)=0$, and when $s_0\geq\sigma_0$ this gives $\dim(\V\cap\varphi(\K))=s_0-\sigma_0$ and $\dim(\V\cap(\K+\varphi(\K)))=s_0-\sigma_0+1$. This implies the result.
    \end{proof}

    % \begin{corollary}\label{cor:the real deal}
    %     The following are equivalent.
    %     \begin{enumerate}
    %         \item $\dim(\A\cap\U)\geq\sigma_0$.
    %         \item $\F^2\subset\B$.
    %         \item $\dim(\A)>\dim(\B)$.
    %         \item $\dim(\B)>\dim(\C)$.
    %         \item $\dim(\A)>\dim(\C)$.
    %     \end{enumerate}
    % \end{corollary}

    \begin{remark}\label{rem:chern mod p}
        Proposition \ref{prop:sigma condition} shows that if $\dim(\A\cap\U)<\sigma_0$ (eg. if $\dim(\A)<\sigma_0$), then the de Rham Chern character restricts to an injection $c_1\otimes k:\A\hookrightarrow\H^2_{\dR}(X)$, and even the Hodge Chern character restricts to an injection $c_1^{\H}\otimes k:\A\hookrightarrow\H^1(X,\Omega^1_X)$.
        %Conversely, if either of these maps are non-injective, then both are non-injective, and $\dim(\A\cap\U)\geq\sigma_0$.
        As a special case, applying this to a single class $L\in\Pic(X)$ shows that the map
        \[
            c_1\otimes\bF_p:\Pic(X)\otimes\bF_p\to\H^2_{\dR}(X)
        \]
        is injective, and that if $\sigma_0\geq 2$ also the map
        \[
            c_1^{\H}\otimes\bF_p:\Pic(X)\otimes\bF_p\to\H^1(X,\Omega^1_X)
        \]
        is injective. This recovers \cite[1.4]{Ogus78}.
    \end{remark}

\begin{pg}
    Let $D\subset X$ be an snc divisor with irreducible components $D_1,\ldots,D_r$ and let $\N_D\subset\N\otimes\bF_p$ denote the $\bF_p$-subspace generated by the images of the Picard classes of the $\ms O(D_i)$ modulo $p$.
\begin{theorem}\label{thm:F2condition}
    The conditions of Theorem \ref{thm:degeneration criterion} hold if and only if $X$ is supersingular and $\dim_{\bF_p}(\N_D\cap\N_0)\geq\sigma_0$.
% are also equivalent to $X$ being supersingular and any (all) of the following holding.
%     \begin{enumerate}
%         %\item The line $\F^2\subset\H^2_{\dR}(X)$ is in the $k$-span of $c_1(L_1),\ldots,c_1(L_r)$.
%         %\item The intersection $\W\cap\F^2$ is nontrivial.
%         \item[(4)] $\dim_k(\A\cap\U)\geq\sigma_0$.
%         \item[(5)] $\dim_{\bF_p}(\T_D\cap\N_0)\geq\sigma_0$.
%     \end{enumerate}
\end{theorem}
\begin{proof}
    Let $\A\subset\N\otimes k$ be the $k$-subspace generated by the classes of the $\ms O(D_i)$ and let $\B\subset\H^2_{\dR}(X)$ be the subspace generated by the $c_1(\ms O(D_i))$. Condition~\eqref{item:3} of \ref{thm:degeneration criterion} is equivalent to the nontriviality of the intersection $\B\cap\F^2$. By \ref{prop:sigma condition} $\B\cap\F^2$ is nontrivial if and only if $s_0\geq\sigma_0$, where $s_0=\dim_k(\A\cap\U)$. We have $\A=\N_D\otimes_{\bF_p}k$ and $\A\cap\U=(\N_D\cap\N_0)\otimes_{\bF_p}k$, so $s_0=\dim_{\bF_p}(\N_D\cap\N_0)$. Combined with \ref{prop:easy cases} to rule out the possibility of $X$ having finite height, this proves the result.
\end{proof}

%As $\N_0$ has dimension $2\sigma_0$ as an $\bF_p$-vector space, there are many choices of $L_1,\ldots,L_r$ satisfying the conditions of \ref{prop:F2condition}. For example, we may take $\T$ to be any $\bF_p$-subspace of $\N_0$ having dimension $r\geq\sigma_0$ and take $L_1,\ldots,L_r\in\Pic(X)$ to be any lifts of a basis for $\T$. For our purposes we additionally want to choose the $L_i$ to be the Picard classes of the irreducible components of an snc divisor. This can always be done, as we now show.

For the existence part of Theorem \ref{thm:main thm 1} we will show that any subspace of $\Pic(X)\otimes\bF_p$ can be realized as the span of the reduction modulo $p$ of the classes of the irreducible components of an snc divisor.

\begin{lemma}\label{lem:very ample lemma}
    For any element $v\in\Pic(X)\otimes\bF_p$ there exists a very ample class in $\Pic(X)$ that is equal to $v$ modulo $p$.
\end{lemma}
\begin{proof}
    Let $L\in\Pic(X)$ be any lift of $v$. Let $A\in\Pic(X)$ be an ample class. By Serre's theorem $L+nA$ is globally generated and $nA$ is very ample for $n\gg 0$. The sum of a globally generated and a very ample class is very ample, so $L+nA$ is very ample for $n\gg 0$. Choosing $n$ to be both sufficiently large and divisible by $p$ gives the desired lift.
\end{proof}

\begin{lemma}\label{lem:choose some curves}
    If $\T\subset\Pic(X)\otimes\bF_p$ is a sub $\bF_p$-vector space of dimension $r$, then there exists an snc divisor $D\subset X$ with exactly $r$ irreducible components, say $D_1,\ldots,D_r$, such that the images of the classes $[\ms O(D_1)],\ldots,[\ms O(D_r)]\in\Pic(X)$ in $\Pic(X)\otimes\bF_p$ are a basis for $\T$.
\end{lemma}
\begin{proof}
    Let $v_1,\ldots,v_r$ be a basis for $\T$. By Lemma \ref{lem:very ample lemma} we may lift this basis to very ample classes $L_1,\ldots,L_r\in\Pic(X)$. For a generic choice of curves $D_1,\ldots,D_r$ in the corresponding linear systems, the divisor $D=\sum_{i}D_i$ will be snc.
\end{proof}
\end{pg}

\begin{proof}[Proof of Theorem \ref{thm:main thm 1}]
    Suppose that $X$ is supersingular of Artin invariant $\sigma_0$. By Theorem \ref{thm:F2condition}, the log HdR spectral sequence for $(X,D)$ is nondegenerate if and only if $\dim_{\bF_p}(\N_D\cap\N_0)\geq\sigma_0$. For the existence, choose an $\bF_p$-subspace $\T\subset\N_0$ of dimension $\sigma_0$ (as $\N_0$ has dimension $2\sigma_0$ we can certainly do this). By Lemma \ref{lem:choose some curves} we may find an snc divisor $D$ having $\sigma_0$ irreducible components $D_1,\ldots,D_{\sigma_0}$ such that the images modulo $p$ of the classes $[\ms O(D_1)],\ldots,[\ms O(D_{\sigma_0})]\in\Pic(X)$ span $\T$. By adding some extra components to $D$ we can also get an snc divisor with this property with number of components equal to any $r\geq\sigma_0$.
\end{proof}

\begin{proof}[Proof of Theorem \ref{thm:Liouville form}]
    Let $D\subset X$ be an snc divisor such that the differential
    \begin{equation}\label{eq:oh gosh}
        \rd:\H^0(X,\Omega^1_X(\log D))\to\H^0(X,\Omega^2_X(\log D))
    \end{equation}
    is nontrivial. Its image is then equal to the subspace $\H^0(X,\Omega^2_X)\subset\H^0(X,\Omega^2_X(\log D))$. Thus, there exists a meromorphic 1-form $\eta\in\H^0(X,\Omega^1_X(\log D))$ such that $\rd\eta=\omega$. The claim following ``Furthermore'' is taken care of by Theorem \ref{thm:main thm 2}, proved below.
\end{proof}

\begin{proof}[Proof of Theorem \ref{thm:symplectic form theorem}]
    Let $D$ and $\eta$ be as in the proof of \ref{thm:Liouville form} above. The restriction $\eta_U$ of $\eta$ to the complement $U=X\setminus D$ is a regular 1-form. The commutativity of the diagram
    \[
        \begin{tikzcd}
            \H^0(X,\Omega^1_X(\log D))\arrow{d}{\rd}\arrow{r}&\H^0(X,\Omega^1_U)\arrow{d}{\rd}\\
            \H^0(X,\Omega^2_X)\arrow{r}&\H^0(U,\Omega^2_U)
        \end{tikzcd}
    \]
    shows that $\rd\eta_U=\omega_U$. So, we have found a nonempty open subset $U\subset X$ such that $\omega_U$ is exact. Now note that we may choose $D$ so that each of its irreducible components moves in a basepoint free linear system (eg. we may choose each component to be very ample). So we can repeat the above three times with generally chosen components of $D$ and get an open cover of $X$ over which $\omega$ is exact. 
\end{proof}

% \begin{proof}[Proof of Theorem \ref{thm:de rham coho must die}]
%     Any class in $\H^1(X,\Omega^1_X)$ or $\H^2(X,\ms O_X)$ can be killed by a Zariski cover. Doing this to a basis gives a cover $U\to X$ that kills both groups in their entirety. By \ref{thm:symplectic form theorem} after taking a further cover we may ensure that the pullback of $\omega$ is exact, and thus has trivial image in $\H^2_{\dR}(U)$. (Note that we are not claiming that the form $\omega\in\H^0(X,\Omega^2_X)$ is itself killed by an open cover; indeed, if $U\subset X$ is any nonempty open subset then the pullback $\H^0(X,\Omega^2_X)\to\H^0(U,\Omega^2_U)$ is injective.)
% \end{proof}

\begin{proof}[Proof of Theorem \ref{thm:main thm 2}]
    The equivalence of~\eqref{item:sc1},~\eqref{item:sc2}, and~\eqref{item:sc3} is implied by \ref{prop:sigma condition}. Note that the equivalence of the latter two is not obvious -- with the notation of the proof of \ref{prop:sigma condition}, what we are deducing is that the composition $\A\twoheadrightarrow\B\twoheadrightarrow\C$ is non-injective if and only if $\B\twoheadrightarrow\C$ is non-injective (even though $\A\twoheadrightarrow\B$ can also be non-injective). The equivalence of~\eqref{item:sc2} and~\eqref{item:sc4} is \ref{thm:degeneration criterion}, and the equivalence of~\eqref{item:sc4} and~\eqref{item:sc5} follows from our description of the only possibly nontrivial differential in the log HdR spectral sequence.

    It remains to show the equivalence of~\eqref{item:sc7} with$~\eqref{item:sc1}-~\eqref{item:sc5}$. By the logarithmic version of Deligne--Illusie, if the pair $(X,D)$ lifts over $W_2$ then the log HdR spectral sequence for $(X,D)$ is degenerate in degrees $<p$ \cite[4.2]{MR894379}, \cite[10.21]{MR1193913} (specifically, this means that if $a+b<p$ then $\mathrm{E}^{a,b}_1=\mathrm{E}^{a,b}_{\infty}$). This shows that$~\eqref{item:sc4}\implies~\eqref{item:sc7}$. On the other hand, we claim that if $D=\sum_{i=1}^rD_i$ is any snc divisor then $(X,D)$ lifts over $W_2$ if and only if there exists a lift of $X$ over $W_2$ together with the line bundles $\ms O(D_1),\ldots,\ms O(D_r)$. It follows from the deformation theory of line bundles on K3 surfaces that the versal deformation space of $(X,\ms O(D_1),\ldots,\ms O(D_r))$ over $W$ is formally smooth if the subspace of $\H^1(X,\Omega^1_X)$ spanned by the Hodge Chern classes of the $\ms O(D_i)$ has dimension equal to $\dim_{\bF_p}(\N_D)$. Thus, this will show that the negation of~\eqref{item:sc3} implies the negation of~\eqref{item:sc7}, hence$~\eqref{item:sc7}\implies~\eqref{item:sc3}$.
    
    We prove the claim. Given a flat deformation $(X',D')$ of $(X,D)$ over $W_2$, each irreducible component of $D'$ is a Cartier divisor in $X'$ and is a flat deformation of a component of $D$. The line bundles associated to the irreducible components of $D'$ give lifts of the $\ms O(D_i)$. Conversely, suppose that we have a flat deformation of $X$ over $W_2$ together with a lift of each $\ms O(D_i)$. For each $i$ we take cohomology of the sequence
    \[
        0\to\ms O(-D_i)\to\ms O_X\to\ms O_{D_i}\to 0
    \]
    and use that $D_i$ is connected and that $\H^1(X,\ms O_X)=0$ to obtain $\H^1(D_i,\ms O(-D_i))=0$. Thus for each $i$ we may lift the section of $\ms O(D_i)$ defining $D_i$ to a section of the lifted line bundle, and so get a deformation of $D$. This completes the proof.
\end{proof}

\begin{remark}
    Here is another proof of \ref{thm:symplectic form theorem}. Suppose $X$ is supersingular of Artin invariant $\sigma_0$. For $m\geq 1$ let $\B_m\subset\Omega^1_X$ denote the $m$th sheaf of higher boundaries. These forms are all closed, so the differential factors through a map $\Omega^1_X/\B_m\to\Omega^2_X$. One can show that the induced map on global sections
    \begin{equation}\label{eq:B!}
        \rd:\H^0(X,\Omega^1_X/\B_{m})\to\H^0(X,\Omega^2_X)
    \end{equation}
    is surjective for $m\geq\sigma_0$ (and an isomorphism for $m=\sigma_0$). Thus, if $m\geq\sigma_0$ then there exists an element $\overline{\eta}\in\H^0(X,\Omega^1_X/\B_m)$ with $\rd\overline{\eta}=\omega$. As $\B_m$ is coherent, classes in $\H^1(X,\B_m)$ die on a Zariski open cover. So, after passing to an open cover we may lift $\overline{\eta}$ to a regular 1-form whose differential is $\omega$.
\end{remark}

\begin{remark}\label{rem:restriction on dr coho}
    Let $X$ be a supersingular K3 surface. Let $\omega\in\H^0(X,\Omega^2_X)$ be a generator and write $[\omega]\in\H^2_{\dR}(X)$ for the de Rham class of $\omega$. By \ref{thm:symplectic form theorem} there exists an open cover $U\to X$ such that $[\omega]$ is in the kernel of the restriction map $\H^2_{\dR}(X)\to\H^2_{\dR}(U)$. For formal reasons, classes in the cohomology groups $\H^1(X,\Omega^1_X)$ and $\H^2(X,\ms O_X)$ also die on a Zariski cover. Thus after refining $U$ we can find a cover $U\to X$ such that the restriction map $\H^2_{\dR}(X)\to\H^2_{\dR}(U)$ induces the zero map on each graded piece $\F^i/\F^{i+1}$ of the Hodge filtration. Nevertheless, as was pointed out to the author by Ambrosi Emiliano, the restriction map $\H^2_{\dR}(X)\to\H^2_{\dR}(U)$ is always nontrivial. This can be seen by looking at the conjugate filtration instead of the Hodge filtration. Specifically, consider the diagram
    \[
        \begin{tikzcd}
            \H^2_{\dR}(X)\arrow[two heads]{d}\arrow{r}&\H^2_{\dR}(U)\arrow{d}\\
            \H^0(X,\ms H^2(\Omega^{\bullet}_X))\arrow[hook]{r}&\H^0(U,\ms H^2(\Omega^{\bullet}_U)).
        \end{tikzcd}
    \]
    The left vertical arrow is surjective because of the degeneration of the conjugate spectral sequence, and the lower horizontal arrow is injective because $\ms H^2(\Omega^{\bullet}_X)$ is a Zariski sheaf. The Cartier operator gives an isomorphism $\ms H^2(\Omega^{\bullet}_X)\simeq\Omega^2_X$, so in particular $\H^0(X,\ms H^2(\Omega^{\bullet}_X))$ is nonzero. It follows that the restriction map $\H^2_{\dR}(X)\to\H^2_{\dR}(U)$ is nonzero. This even shows something stronger, namely that this map is nontrivial for any nonempty open subset $U\subset X$.

\end{remark}

\begin{remark}\label{rem:exact 2 form}
    The phenomenon in Theorem \ref{thm:symplectic form theorem} cannot happen in characteristic zero. Indeed, suppose that $X$ is any smooth proper variety over an algebraically closed field of characteristic 0 and let $\omega\in\H^0(X,\Omega^2_X)$ be a nonzero 2-form. By the degeneration of the Hodge--de Rham spectral sequence we have $\H^0(X,\Z\Omega^2_X)=\H^0(X,\Omega^2_X)$ and $\omega$ is automatically closed. We claim that the restriction of $\omega$ to any nonempty open subset $U\subset X$ cannot be exact. To see this, let $D=X\setminus U$ be the complement. By shrinking $U$ we may assume that $U$ is affine and $D$ is a divisor, and by taking an embedded resolution we may assume that $D$ is snc. Let $\Q^{\bullet}$ be the quotient in the sequence
    \begin{equation}\label{eq:Q SES}
        0\to\Omega^{\bullet}_X\to\Omega^{\bullet}_X(\log D)\to\Q^{\bullet}\to 0.
    \end{equation}
    We have $\Q^0=0$, and~\eqref{eq:residue SES} shows that $\Q^1\simeq\bigoplus_{i=1}^r\ms O_{D_i}$, where $D_1,\ldots,D_r$ are the irreducible components of $D$. Because the characteristic is zero we have $\Z\ms O_{D_i}=\underline{k}_{D_i}$. It follows that 
    \[
        \H^1(X,\Q^{\bullet})=\H^0(X,\Z\Q^1)=\H^0(X,\textstyle\bigoplus_i\Z\ms O_{D_i})=\H^0(X,\textstyle\bigoplus_i\underline{k}_{D_i})=k^{\oplus r}.
    \]
    Taking cohomology of~\eqref{eq:Q SES} and using that the natural map $\Omega^{\bullet}_X(\log D)\to j_*\Omega^{\bullet}_U$ is a quasi-isomorphism \cite[II 3.13]{MR417174} (where $j:U\hookrightarrow X$ is the inclusion), we get an exact diagram
    \begin{equation}\label{eq:master diagram}
        \begin{tikzcd}
            &&\H^0(U,\Omega^1_U)\arrow{d}{\rd}\\
            &\H^0(X,\Omega^2_X)\arrow[hook]{d}\arrow{r}&\H^0(U,\Z\Omega^2_U)\arrow{d}\\
           k^{\oplus r}\arrow{r}{\delta}&\H^2_{\dR}(X)\arrow{r}&\H^2_{\dR}(U).
        \end{tikzcd}
    \end{equation}
    The exactness of the right column follows from taking cohomology of the short exact sequence of complexes
    \[
        0\to[\ms O_U\xrightarrow{\rd}\Omega^1_U]\to\Omega^{\bullet}_U\to\Omega^{\geq 2}_U\to 0
    \]
    and noting that $\H^2(U,\Omega^{\geq 2}_U)=\H^0(U,\Z\Omega^2_U)$ and that $\H^1(U,\ms O_U)=0$ (we assumed $U$ affine). It follows from \ref{lem:delta} that the boundary map $\delta$ is given by $(\lambda_1,\ldots,\lambda_r)\mapsto\sum_{i}\lambda_ic_1(\ms O(D_i))$. By the Lefschetz principle we may assume that the ground field is the complex numbers. Thus the image of $\delta$ is contained in the subspace of classes of type $(1,1)$. The de Rham class of $\omega$ is of type $(2,0)$ and so is not in the image of $\delta$. Therefore the de Rham class of $\omega_U$ is nontrivial, so by the exactness of the right column $\omega_U$ is not exact.
\end{remark}

\bibliographystyle{plain}
\bibliography{biblio}

\end{document}